\documentclass[a4, 12pt]{amsart}
\usepackage[mathscr]{eucal}
\usepackage{amssymb}
\usepackage{latexsym}
\usepackage{amsthm}
\usepackage{color}
\theoremstyle{plain}
\newtheorem{theorem}{Theorem}[section]

\newtheorem{remark}{Remark}[section]

\newcommand{\Ric}{\mathrm{Ric}}

\newtheorem{theo}{Theorem}[section]

	\newtheorem{lem}[theo]{Lemma}
	\newtheorem{prop}[theo]{Proposition}

\newtheorem{claim}{Claim}[section]

\newcommand{\be}{\begin{equation}}
	\newcommand{\ee}{\end{equation}}
\newcommand{\ba}{\begin{eqnarray}}
	\newcommand{\ea}{\end{eqnarray}}
\newcommand{\ban}{\begin{eqnarray*}}
	\newcommand{\ean}{\end{eqnarray*}}

\makeatletter
\@addtoreset{equation}{section}

 \title[Shrinking gradient Ricci soliton]{Rigidity  of shrinking gradient ricci soliton with constant scalar curvature}

\author{Fengjiang Li}
\address[Fengjiang Li]
{Mathematical Science Research Center, Chongqing University of Technology, Chongqing 400054, China}
\email{fengjiangli@cqut.edu.cn}

\author{Yuanyuan Qu}
\address[Yuanyuan Qu]{School of Mathematical Sciences, Shanghai Key Laboratory of PMMP, East China Normal University, Shanghai 200241,
	China}
\email{52285500012@stu.ecnu.edu.cn}

\author{Guoqiang Wu}
\address[Guoqiang Wu]
{School of Science, Zhejiang Sci-Tech University, Hangzhou 310018, China}
\email{gqwu@zstu.edu.cn}


\begin{document}
\maketitle

\begin{abstract}
Let $(M^n, g, f)$ be a complete shrinking gradient Ricci soliton with constant scalar curvature. Under the assumptions that
	\begin{enumerate}
	\item[(i)] $\Ric \geq \dfrac{\nabla_{\nabla f}\Ric}{f}$ on $M \setminus D$, where $D$ is a compact subset of $M$;
	\item[(ii)] $(M^n, g, f)$ smoothly converges to $\mathbb{R}^2 \times \mathbb{S}^{n-2}$,
\end{enumerate}
we conclude that $(M^n, g, f)$ is isometric to $\mathbb{R}^2 \times \mathbb{S}^{n-2}$. Notably, condition \textup{(i)} is weaker than the radial flatness condition in \cite{Petersen-Wylie2}.
\end{abstract}
\maketitle

\section{Introduction}
Recall that a gradient Ricci soliton is a triple $(M^n, g, f)$ satisfying
\begin{equation}\label{1.1}
	\Ric + \nabla^2 f = \lambda g
\end{equation}
for some constant $\lambda \in \mathbb{R}$. We say that the soliton is shrinking, steady, or expanding according as $\lambda > 0$, $= 0$, or $< 0$, respectively. By rescaling the metric $g$ and changing $f$ by a constant, we may assume that $\lambda \in \{-\frac{1}{2}, 0, \frac{1}{2}\}$.

A gradient Ricci soliton is a self-similar solution to the Ricci flow that evolves by diffeomorphisms and homotheties. The study of solitons has become increasingly important in both the study of the Ricci flow and metric measure theory. Solitons play an essential role as singularity models in the Ricci flow proof of uniformization. Due to the work of Perelman \cite{Perelman2}, Ni--Wallach \cite{Ni-Wallach}, and Cao--Chen--Zhu \cite{Cao-Chen-Zhu}, the classification of three-dimensional shrinking gradient Ricci solitons is complete. For more work on the classification of gradient Ricci solitons under various curvature conditions, see \cite{Brendle1, Cao-Chen, Cao-Chen-Zhu, Cao-Chen2, Cao-Wang-Zhang, Chen-Wang, Eminenti-LaNave-Mantegazza, Kotschwar, Naber, Petersen-Wylie, Pigola-Rimoldi-Setti, Wu-Wu-Wylie, Zhang}.

In this paper, we focus our attention on $n$-dimensional gradient shrinking Ricci solitons with constant scalar curvature. Recall that in Petersen--Wylie \cite{Petersen-Wylie}, a gradient Ricci soliton $(M, g)$ is said to be rigid if it is isometric to a quotient of ${N} \times \mathbb{R}^k$, the product soliton of an Einstein manifold ${N}$ of positive scalar curvature with the Gaussian soliton $\mathbb{R}^k$.

Furthermore, for the complete shrinking case, Professor Huai-Dong Cao conjectured that $(M^n, g, f)$ has constant scalar curvature if and only if it is rigid, i.e., a finite quotient of ${N}^k \times \mathbb{R}^{n-k}$ for some Einstein manifold ${N}$ of positive scalar curvature.

Regarding this conjecture, Fern\'{a}ndez-L\'{o}pez--Garc\'{i}a-R\'{i}o \cite{FR10} and Munteanu--Sesum \cite{Munteanu-Sesum} proved that $n$-dimensional complete gradient shrinking solitons with harmonic Weyl tensor are rigid. Catino--Mastrolia--Monticelli \cite{CMM17} showed that any gradient shrinking Ricci soliton with fourth-order divergence-free Weyl tensor is rigid.

Petersen and Wylie \cite{Petersen-Wylie} proved that a complete gradient Ricci soliton is rigid if and only if it has constant scalar curvature and is radially flat, that is, the sectional curvature $K(\cdot, \nabla f) = 0$. Fern\'{a}ndez-L\'{o}pez and Garc\'{i}a-R\'{i}o \cite{FR16} obtained that the soliton is rigid if and only if the Ricci curvature has constant rank. They also derived the following results for complete $n$-dimensional gradient Ricci solitons satisfying \eqref{1.1} with constant scalar curvature $R$:
(i) The possible values of $R$ are $\{0, \lambda, \dots, (n-1)\lambda, n\lambda\}$.
(ii) If $R$ takes the value $(n-1)\lambda$, then the soliton must be rigid.
(iii) In the shrinking case, there is no complete gradient shrinking Ricci soliton with $R = \lambda$.
(iv) Any $n$-dimensional gradient shrinking Ricci soliton with constant scalar curvature $R = (n-2)\lambda$ has nonnegative Ricci curvature.

Several years ago, Cheng and Zhou \cite{Cheng-Zhou} confirmed Cao's conjecture in dimension $n = 4$. Very recently, the authors gave a simple proof of Cheng--Zhou's result in \cite{Ou-Qu-Wu}. Later, the authors \cite{Li-Ou-Qu-Wu} finished the $5$-dimensional case when the scalar curvature is $\frac{3}{2}$. In this paper, we seek to generalize these results to higher dimensions in a certain sense.

\begin{theorem}\label{thm:main}
	Suppose $(M^n, g, f)$ is a complete shrinking gradient Ricci soliton with constant scalar curvature. Assume that
	\begin{enumerate}
		\item[(i)] $\Ric \geq \dfrac{\nabla_{\nabla f}\Ric}{f}$ on $M \setminus D$, where $D$ is a compact subset of $M$;
		\item[(ii)] $(M^n, g, f)$ smoothly converges to $\mathbb{R}^2 \times \mathbb{S}^{n-2}$.
	\end{enumerate}
	Then $(M^n, g, f)$ is isometric to $\mathbb{R}^2 \times \mathbb{S}^{n-2}$.
\end{theorem}

\begin{remark}
	It is easy to see that the scalar curvature must be $ (n-2)\lambda$, since $(M^n, g, f)$ smoothly converges to $\mathbb{R}^2 \times \mathbb{S}^{n-2}$.
\end{remark}

\begin{remark}
	Condition \textup{(i)} is weaker than the radial flatness condition in \cite{Petersen-Wylie2}. To see this, we take $\lambda=\frac{1}{2}$ and suppose $\{e_1, \dots, e_n\}$ are the eigenvectors of Ricci curvature corresponding to eigenvalues $\lambda_1 \leq \dots \leq \lambda_n$, where $\lambda_1 = 0$ and $e_1 = \frac{\nabla f}{|\nabla f|}$. Recall that radial flatness means $R(\nabla f, \cdot, \nabla f, \cdot) = 0$. Given $v = \sum_{i=2}^n a_i e_i \in T\Sigma(t)$, where $\Sigma(t) = f^{-1}(t)$, by \eqref{Ricci derivative} we have
	\begin{align*}
		&\Ric(v, v) - \frac{\nabla_{\nabla f}\Ric}{f}(v, v) \\
		&\quad = \Ric(v, v) - \frac{\Ric \circ (\Ric - \frac{1}{2}g)(v, v)}{f} \\
		&\quad = \sum_{i=2}^n a_i^2 \lambda_i - \sum_{i=2}^n \frac{a_i^2 \lambda_i(\lambda_i - \frac{1}{2})}{t} \\
		&\quad = \sum_{i=2}^n a_i^2 \lambda_i\left(1 - \frac{\lambda_i - \frac{1}{2}}{t}\right) \geq 0
	\end{align*}
	provided that $t$ is sufficiently large, where we have used $|\Ric|^2 = \frac{n-2}{4}$ and $\Ric \geq 0$ from Lemma~\ref{ricci nonnegative}.
\end{remark}


\section{Preliminary}

Suppose $(M^n, g, f)$ is a shrinking gradient Ricci soliton satisfying
\begin{equation}\label{soliton-equation}
	\nabla^2 f + \mathrm{Ric} = \frac{1}{2}g
\end{equation}
with constant scalar curvature.

We first recall some basic formulas that will be used throughout this paper:
\begin{align}
	&dR = 2\,\mathrm{Ric}(\nabla f), \label{second-bianchi}\\
	&R + \Delta f = \frac{n}{2}, \label{trace-soliton}\\
	&R + |\nabla f|^2 = f, \label{f-gradient}\\
	&\Delta_f R = R - 2|\mathrm{Ric}|^2, \label{drift-laplacian-scalar}\\
	&\Delta_f R_{ij} = R_{ij} - 2R_{ikjl}R_{kl}, \label{elliptic-equation}
\end{align}
where $\Delta_f \mathrm{Ric} = \Delta \mathrm{Ric} - \nabla_{\nabla f}\mathrm{Ric}$ in the last formula.

In our setting, since the scalar curvature is constant, the potential function $f$ can be renormalized by replacing $f - R$ with $f$, so that $f: M \to [0, +\infty)$ satisfies
\begin{equation}\label{iso1}
	|\nabla f|^2 = f.
\end{equation}

Since $\mathrm{Ric}(\nabla f) = \frac{1}{2}dR = 0$, we choose $e_1 = \frac{\nabla f}{|\nabla f|}$. Let $\{e_1, e_2, e_3, \ldots, e_n\}$ be a local orthonormal frame at any point $p \in M \setminus f^{-1}(0)$ such that $R_{\alpha\beta} = \mathrm{Ric}(e_\alpha, e_\beta) = \lambda_\alpha \delta_{\alpha\beta}$. Denote $\Sigma(s) = f^{-1}(s)$; when clear from the context, we do not distinguish between $\Sigma(s)$ and $\Sigma$.

Recall that the intrinsic curvature tensor $R^{\Sigma}_{\alpha\beta\gamma\eta}$ and the extrinsic curvature $R_{\alpha\beta\gamma\eta}$ of $p \in \Sigma$ (where $\{\alpha,\beta,\gamma,\eta\} \subset \{2, \ldots, n\}$) are related by the Gauss equations:
\begin{equation*}\label{gauss-equation}
	R^{\Sigma}_{\alpha\beta\gamma\eta} = R_{\alpha\beta\gamma\eta} + h_{\alpha\gamma}h_{\beta\eta} - h_{\alpha\eta}h_{\beta\gamma},
\end{equation*}
where $h_{\alpha\beta}$ denote the components of the second fundamental form of $\Sigma$, i.e.,
\begin{equation*}\label{second-fundamental-form}
	h_{\alpha\beta} = \frac{\frac{1}{2} - \lambda_\alpha}{\sqrt{f}}\delta_{\alpha\beta}.
\end{equation*}
Moreover,
\begin{equation*}\label{ricci-level-set}
	R^{\Sigma}_{\alpha\gamma} = R_{\alpha\gamma} - R_{\alpha 1\gamma 1} + Hh_{\alpha\gamma} - h_{\alpha\beta}h_{\beta\gamma},
\end{equation*}
and the scalar curvature $R^{\Sigma}$ of $\Sigma$ satisfies
\begin{equation*}\label{scalar-level-set}
	R^{\Sigma} = R - 2R_{11} + H^2 - |A|^2.
\end{equation*}
Since $R = \frac{n-2}{2}$, $\mathrm{Ric}(\nabla f) = 0$, and $R_{1i} = 0$ for $i = 1, \ldots, n$, we have
\begin{equation*}\label{scalar-simplified}
	R^{\Sigma} = R + H^2 - |A|^2.
\end{equation*}

Since the mean curvature $H = \frac{\frac{n-1}{2} - \sum_{\alpha=2}^n \lambda_\alpha}{\sqrt{f}} = \frac{1}{2\sqrt{f}}$, we obtain
\begin{align*}
	R^{\Sigma} &= R + H^2 - |A|^2 \notag\\
	&= \frac{n-2}{2} + \frac{1}{4f} - \frac{1}{f}\sum_{\alpha=2}^n\left(\frac{1}{2} - \lambda_\alpha\right)^2 \notag\\
	&= \frac{n-2}{2} + \frac{1}{4f} - \frac{1}{f}\left(\frac{n-1}{4} - \sum_{\alpha=2}^n\lambda_\alpha + \sum_{\alpha=2}^n\lambda_\alpha^2\right) \notag\\
	&= \frac{n-2}{2} + \frac{1}{4f} - \frac{1}{f}\left(\frac{n-1}{4} - \frac{n-2}{2} + \frac{n-2}{4}\right) \notag\\
	&= \frac{n-2}{2}. \label{scalar-final}
\end{align*}
In the above, we have used $\lambda_1 = 0$, $R = \frac{n-2}{2}$, and $|\mathrm{Ric}|^2 = \frac{n-2}{4}$.

Recall the decomposition of the Riemann curvature tensor:
\begin{equation*}\label{riemann-decomposition}
	R_{ijkl} = W_{ijkl} + \frac{1}{n-2}(g_{ik}R_{jl} - g_{il}R_{jk} - g_{jk}R_{il} + g_{jl}R_{ik}) - \frac{1}{(n-1)(n-2)}R(g_{ik}g_{jl} - g_{il}g_{jk}),
\end{equation*}
we obtain
\begin{equation*}
\begin{aligned}
	& R^{\Sigma}_{\alpha\beta\alpha\beta} \notag\\
	=& W^{\Sigma}_{\alpha\beta\alpha\beta} + \frac{1}{n-3}(R^{\Sigma}_{\alpha\alpha} + R^{\Sigma}_{\beta\beta}) - \frac{1}{(n-2)(n-3)}R^{\Sigma} \notag\\
	=& W^{\Sigma}_{\alpha\beta\alpha\beta} + \frac{1}{n-3}\left(R_{\alpha\alpha} - R_{\alpha 1\alpha 1} + Hh_{\alpha\alpha} - h_{\alpha\alpha}^2 + R_{\beta\beta} - R_{\beta 1\beta 1} + Hh_{\beta\beta} - h_{\beta\beta}^2\right) - \frac{1}{2(n-3)} \notag\\
	=& W^{\Sigma}_{\alpha\beta\alpha\beta} + \frac{1}{n-3}\left(\lambda_{\alpha} + \lambda_{\beta} - R_{\alpha 1\alpha 1} - R_{\beta 1\beta 1} + H(h_{\alpha\alpha} + h_{\beta\beta}) - h_{\alpha\alpha}^2 - h_{\beta\beta}^2\right) - \frac{1}{2(n-3)}. \label{sectional-level-set}
\end{aligned}
\end{equation*}
Using the Gauss equation together with $h_{\alpha\beta} = 0$ for $\alpha \neq \beta$, we have
\begin{equation*}
	\begin{aligned}
	& R_{\alpha\beta\alpha\beta} \notag\\
	= &R^{\Sigma}_{\alpha\beta\alpha\beta} - h_{\alpha\alpha}h_{\beta\beta} + h_{\alpha\beta}^2 \notag\\
	=& \frac{1}{n-3}\left(\lambda_{\alpha} + \lambda_{\beta} - R_{\alpha 1\alpha 1} - R_{\beta 1\beta 1} + H(h_{\alpha\alpha} + h_{\beta\beta}) - h_{\alpha\alpha}^2 - h_{\beta\beta}^2\right) - \frac{1}{2(n-3)} \notag\\
	&\quad - h_{\alpha\alpha}h_{\beta\beta} + W^{\Sigma}_{\alpha\beta\alpha\beta} \notag\\
	=& \frac{1}{n-3}(\lambda_{\alpha} + \lambda_{\beta}) - \frac{1}{2(n-3)} - \frac{1}{n-3}(R_{\alpha 1\alpha 1} + R_{\beta 1\beta 1}) \notag\\
	& + \frac{1}{2(n-3)f}\left[\left(\frac{1}{2} - \lambda_{\alpha}\right) + \left(\frac{1}{2} - \lambda_{\beta}\right)\right] \notag\\
	&- \frac{\left(\frac{1}{2} - \lambda_{\alpha}\right)^2 + \left(\frac{1}{2} - \lambda_{\beta}\right)^2}{(n-3)f} - \frac{\left(\frac{1}{2} - \lambda_{\alpha}\right)\left(\frac{1}{2} - \lambda_{\beta}\right)}{f} + W^{\Sigma}_{\alpha\beta\alpha\beta}. \label{sectional-final}
\end{aligned}
\end{equation*}

\begin{claim}
\begin{equation}\label{ricci-curvature-level-set}
	\mathrm{Ric}^{\Sigma} = \mathrm{Ric} - \frac{\nabla_{\nabla f}\mathrm{Ric}}{f}.
\end{equation}
\end{claim}

Since on a Ricci shrinker with constant scalar curvature,
\begin{equation}\label{Ricci derivative}
	\nabla_{\nabla f}\mathrm{Ric} = \mathrm{Ric} \circ \left(\mathrm{Ric} - \frac{1}{2}g\right) + R(\nabla f, \cdot, \nabla f, \cdot),
\end{equation}
we have
\begin{equation}\label{R1a1b}
	R(e_1, e_\alpha, e_1, e_\beta) = \frac{\nabla_{\nabla f}R_{\alpha\beta} + \left(\frac{1}{2}R_{\alpha\beta} - \sum_{k=1}^n R_{\alpha k}R_{k\beta}\right)}{f}.
\end{equation}
Finally,
\begin{equation*}\label{ricci-level-set-proof}
	\begin{aligned}
		R^{\Sigma}_{\alpha\beta} =& R_{\alpha\beta} - \frac{\nabla_{\nabla f}R_{\alpha\beta} + \left(\frac{1}{2}R_{\alpha\beta} - R_{\alpha k}R_{k\beta}\right)}{f} + \frac{\frac{1}{2} - \lambda_\alpha}{2f}\delta_{\alpha\beta}\\
		& - \frac{\left(\frac{1}{2} - \lambda_\alpha\right)\left(\frac{1}{2} - \lambda_\beta\right)}{f}\delta_{\alpha\gamma}\delta_{\gamma\beta} \\
		=& R_{\alpha\beta} - \frac{\nabla_{\nabla f}R_{\alpha\beta}}{f} + \frac{1}{f}\left[-\lambda_\alpha\left(\frac{1}{2} - \lambda_\alpha\right) + \frac{1}{2}\left(\frac{1}{2} - \lambda_\alpha\right) - \left(\frac{1}{2} - \lambda_\alpha\right)^2\right]\delta_{\alpha\beta} \\
		=& R_{\alpha\beta} - \frac{\nabla_{\nabla f}R_{\alpha\beta}}{f}.
	\end{aligned}
\end{equation*}

\begin{lem}\label{ricci nonnegative}
	Let $(M^n, g, f)$ be an $n$-dimensional shrinking gradient Ricci soliton with constant scalar curvature $\frac{n-2}{2}$. Then $\mathrm{Ric} \geq 0$.
\end{lem}
\begin{proof}
	This fact is known in the literature; for more details, see \cite{FR16}.
\end{proof}

\begin{lem}\label{lambda-estimates}
	Let $(M^n, g, f)$ be an $n$-dimensional shrinking gradient Ricci soliton with constant scalar curvature $\frac{n-2}{2}$. Then
	\begin{equation}\label{la}
		\left(\frac{1}{2} - \lambda_\alpha\right)^2 \leq \sum_{\alpha=3}^n\left(\frac{1}{2} - \lambda_\alpha\right)^2 = \lambda_2(1 - \lambda_2)
	\end{equation}
	for $\alpha = 3, \ldots, n$;
	\begin{equation}\label{l345}
		\sum_{\alpha=3}^n\left(\frac{1}{2} - \lambda_\alpha\right)^2\lambda_\alpha \leq \lambda_2(1 - \lambda_2)\left(\frac{n-2}{2} - \lambda_2\right);
	\end{equation}
	\begin{equation}\label{l345'}
		\sum_{\alpha=3}^n\left(\frac{1}{2} - \lambda_\alpha\right)^2\lambda_\alpha^2 \leq \lambda_2(1 - \lambda_2)\left(\frac{n-2}{4} - \lambda_2^2\right).
	\end{equation}
\end{lem}

\begin{proof}
	It follows from \eqref{drift-laplacian-scalar} that
	\begin{equation*}\label{sum-identity}
		\begin{aligned}
			0 &= \sum_{\alpha=2}^n\lambda_\alpha\left(\frac{1}{2} - \lambda_\alpha\right) \\
			&= -\sum_{\alpha=2}^n\left[\left(\frac{1}{2} - \lambda_\alpha\right)^2 + \frac{1}{2}\left(\frac{1}{2} - \lambda_\alpha\right)\right] \\
			&= -\sum_{\alpha=2}^n\left(\frac{1}{2} - \lambda_\alpha\right)^2 + \frac{1}{4},
		\end{aligned}
	\end{equation*}
	which yields
	\begin{equation*}\label{lambda2-identity}
		\begin{aligned}
			\sum_{\alpha=3}^n\left(\frac{1}{2} - \lambda_\alpha\right)^2 &= \frac{1}{4} - \left(\frac{1}{2} - \lambda_2\right)^2 \\
			&= \lambda_2(1 - \lambda_2).
		\end{aligned}
	\end{equation*}
	Therefore, we have
	\begin{equation*}\label{l345-proof}
		\begin{aligned}
			\sum_{\alpha=3}^n\left(\frac{1}{2} - \lambda_\alpha\right)^2\lambda_\alpha &\leq \lambda_2(1 - \lambda_2)(\lambda_3 + \cdots + \lambda_n) \\
			&= \lambda_2(1 - \lambda_2)\left(\frac{n-2}{2} - \lambda_2\right).
		\end{aligned}
	\end{equation*}
	Similarly, \eqref{l345'} holds.
\end{proof}

\begin{lem}\label{sectional-curvature}
	Let $(M^n, g, f)$ be an $n$-dimensional shrinking gradient Ricci soliton with constant scalar curvature. Then
	\begin{equation}\label{k1a}
		K_{1\alpha} = \frac{\nabla_{\nabla f}R_{\alpha\alpha} + \lambda_\alpha\left(\frac{1}{2} - \lambda_\alpha\right)}{f}
	\end{equation}
	for $\alpha = 2, \ldots, n$.
\end{lem}
\begin{proof}
	This follows directly from \eqref{R1a1b}.
\end{proof}

\section{Weyl curvature integral pinching}
In this section, we prove that the $L^2$ norm of the Weyl curvature can be controlled by the $L^2$ norm of the traceless Ricci curvature, provided that the metric has constant scalar curvature and is sufficiently close to the standard round metric on $\mathbb{S}^n$.

We first derive the following theorem using elliptic estimates.

\begin{theo}\label{thm:pinching-estimate}
	Suppose $g_c$ is the standard round metric on $\mathbb{S}^n$ with scalar curvature $n(n-1)$. Then there exist constants $C > 0$ and $\delta > 0$ such that if $g$ is any metric on $\mathbb{S}^n$ satisfying $\|g - g_c\|_{C^{100}(g_c)} \leq \delta$, then
	\[
	\int_{\mathbb{S}^n} |\nabla \mathring{Ric}|^2 \, dvol_g \leq C \int_{\mathbb{S}^n} |\mathring{Ric}|^2 \, dvol_g.
	\]
\end{theo}

\begin{proof}
	Throughout the proof, the value of $C$ may change from line to line, but depends only on $\delta$ and $n$. First, integrating by parts, we obtain
	\begin{align*}
		\int_{\mathbb{S}^n} |\nabla \mathring{Ric}|^2 \, dvol_g
		&= -\int_{\mathbb{S}^n} \langle \mathring{Ric}, \Delta \mathring{Ric} \rangle \, dvol_g \\
		&\leq \|\mathring{Ric}\|_{L^2(\mathbb{S}^n, g)} \|\Delta \mathring{Ric}\|_{L^2(\mathbb{S}^n, g)}.
	\end{align*}
	On the other hand,
	\begin{align*}
		\|\Delta \mathring{Ric}\|_{L^2(\mathbb{S}^n, g)}^2
		&= \int_{\mathbb{S}^n} \langle \Delta \mathring{Ric}, \Delta \mathring{Ric} \rangle \, dvol_g \\
		&= \int_{\mathbb{S}^n} \langle \mathring{Ric}, \Delta^2 \mathring{Ric} \rangle \, dvol_g \\
		&\leq C \|\mathring{Ric}\|_{L^2(\mathbb{S}^n, g)} \left(\|\mathring{Ric}\|_{L^2(\mathbb{S}^n, g)} + \|\Delta \mathring{Ric}\|_{L^2(\mathbb{S}^n, g)}\right) \\
		&\leq C \|\mathring{Ric}\|_{L^2(\mathbb{S}^n, g)}^2 + \frac{1}{2} \|\Delta \mathring{Ric}\|_{L^2(\mathbb{S}^n, g)}^2,
	\end{align*}
	where in the third inequality we use the elliptic $W^{2,2}$ estimate for the Laplacian operator, and the constant $C$ depends on $\delta$ and the dimension $n$. By absorption, we derive that
	\[
	\|\Delta \mathring{Ric}\|_{L^2(\mathbb{S}^n, g)}^2 \leq C \|\mathring{Ric}\|_{L^2(\mathbb{S}^n, g)}^2.
	\]
	Combining this inequality with the first one, we finally obtain
	\[
	\|\nabla \mathring{Ric}\|_{L^2(\mathbb{S}^n, g)}^2 \leq C \|\mathring{Ric}\|_{L^2(\mathbb{S}^n, g)}^2. \qedhere
	\]
\end{proof}

Next, in order to control the Weyl curvature, we need to consider the derivative of $|W|^2$. Actually,  we have the following basic formula.

\begin{lem}[\cite{Catino}]\label{lem:catino-estimate}
	Let $(M^n, g)$, $n \geq 4$, be an $n$-dimensional Riemannian manifold. Then
	\[
	\frac{1}{2}\Delta |W|^2 = |\nabla W|^2 - 2\frac{n-2}{n-3}|\delta W|^2 + 2R_{pq}W_{pikl}W_{qikl} - 3W_{ijkl}W_{ijpq}W_{pqkl} - 2(W_{ijkl}C_{jkl})_i,
	\]
	where $\delta W$ is the divergence of the Weyl curvature, and $C$ is the Cotton tensor, i.e.,
	\[
	C_{ijk} = \nabla_k R_{ij} - \nabla_j R_{ik} - \frac{1}{2(n-1)}(R_k g_{ij} - R_j g_{ik}).
	\]
	Moreover,
	\[
	(\delta W)_{ijk} = \nabla_l W_{lijk} = -\frac{n-3}{n-2}C_{ijk}.
	\]
\end{lem}

Now we can state the main result of this section.

\begin{theo}\label{thm:weyl-curvature-estimate}
	Suppose $g_c$ is the standard round metric on $\mathbb{S}^n$ with scalar curvature $n(n-1)$. Then there exist constants $C > 0$ and $\delta > 0$ such that if $g$ is any metric on $\mathbb{S}^n$ with constant scalar curvature $n(n-1)$ satisfying $\|g - g_c\|_{C^{100}(g_c)} \leq \delta$, then
	\[
	\int_{\mathbb{S}^n} |W|^2 \, dvol_g \leq C \int_{\mathbb{S}^n} |\mathring{Ric}|^2 \, dvol_g,
	\]
	where $C$ depends only on $\delta$ and $n$.
\end{theo}

\begin{proof}
	Since $\|g - g_c\|_{C^{100}(g_c)} \leq \delta$, we have
	\[
	R_{pq}W_{pikl}W_{qikl} \geq (n-1-\beta)|W|^2
	\]
	and
	\[
	|W_{ijkl}W_{ijpq}W_{pqkl}| \leq \beta|W|^2,
	\]
	where $\beta$ is a sufficiently small positive constant depending on $\delta$ and $n$. Notice that
	\[
	|\delta W|^2 = \frac{(n-3)^2}{(n-2)^2}|C|^2 \leq 4 \cdot \frac{(n-3)^2}{(n-2)^2}|\nabla Ric|^2 = 4 \cdot \frac{(n-3)^2}{(n-2)^2}|\nabla \mathring{Ric}|^2,
	\]
	since $g$ has constant scalar curvature. Now, integrating the identity in Lemma \ref{lem:catino-estimate}, we obtain
	\[
	\int_{\mathbb{S}^n} \bigl(2(n-1-\beta) - 3\beta\bigr)|W|^2 \, dvol_g \leq 2\frac{n-2}{n-3} \cdot 4 \cdot \frac{(n-3)^2}{(n-2)^2} \int_{\mathbb{S}^n} |\nabla \mathring{Ric}|^2 \, dvol_g.
	\]
	Combining this with Theorem \ref{thm:pinching-estimate}, we get
	\[
	\int_{\mathbb{S}^n} |W|^2 \, dvol_g \leq C \int_{\mathbb{S}^n} |\nabla \mathring{Ric}|^2 \, dvol_g \leq C \int_{\mathbb{S}^n} |\mathring{Ric}|^2 \, dvol_g,
	\]
	where $C$ depends only on $\delta$ and $n$.
\end{proof}

\section{A key estimate of $|\nabla \mathrm{Ric}|^2$}\label{sec:key-estimate}

For notational simplicity, we denote $K_{\alpha\beta} = R(e_\alpha, e_\beta, e_\alpha, e_\beta)$ and $W^\Sigma_{\alpha\beta} = W^\Sigma(e_\alpha, e_\beta, e_\alpha, e_\beta)$.

\begin{lem}\label{lem:ricci-gradient}
	Let $(M^n, g, f)$ be an $n$-dimensional shrinking gradient Ricci soliton with constant scalar curvature $R = \frac{n-2}{2}$. Then
	\begin{equation*}\label{eq:ricci-gradient}
		|\nabla \mathrm{Ric}|^2 = 2\sum_{\substack{\alpha,\beta=2 \\ \alpha\neq\beta}}^n K_{\alpha\beta}\lambda_{\alpha}\lambda_{\beta} - \frac{n-2}{4}.
	\end{equation*}
\end{lem}

\begin{proof}
	Notice that $|\mathrm{Ric}|^2 = \frac{1}{2}R = \frac{n-2}{4}$. We have
	\[
	\frac{1}{2}\Delta_f|\mathrm{Ric}|^2 = R_{ij}\Delta_f R_{ij} + |\nabla \mathrm{Ric}|^2,
	\]
	and therefore
\begin{equation*}
	\begin{aligned}
		|\nabla \mathrm{Ric}|^2 &= -\sum_{i,j=1}^n R_{ij}\Delta_f R_{ij} \\
		&= \sum_{i,j=1}^n R_{ij}(2R_{ikjl}R_{kl} - R_{ij}) \\
		&= \sum_{\substack{\alpha,\beta=2 \\ \alpha\neq\beta}}^n 2R_{\alpha\beta\alpha\beta}\lambda_\alpha\lambda_\beta - |\mathrm{Ric}|^2 \\
		&= \sum_{\substack{\alpha,\beta=2 \\ \alpha\neq\beta}}^n 2K_{\alpha\beta}\lambda_{\alpha}\lambda_{\beta} - \frac{n-2}{4}. \qedhere
	\end{aligned}
\end{equation*}
\end{proof}

Now we can state the key estimate for $|\nabla \mathrm{Ric}|^2$; the five-dimensional case was derived in \cite{Li-Ou-Qu-Wu}.

\begin{prop}\label{prop:key-estimate}
	Let $(M^n, g, f)$ be an $n$-dimensional shrinking gradient Ricci soliton with constant scalar curvature $R = \frac{n-2}{2}$. Suppose the assumptions in Theorem \ref{thm:main} hold. Then there exists a constant $C = C(n) > 0$ such that
	\[
	|\nabla \mathrm{Ric}|^2 \leq -\frac{1}{n-3}(\lambda_1+\lambda_2) + K_{12} + C|W^{\Sigma}|^2
	\]
	on $M \setminus D(a)$ for some $a > 0$.
\end{prop}

\begin{proof}
	Recall that
\begin{equation*}
	\begin{aligned}
		K_{\alpha\beta} &= K^{\Sigma}_{\alpha\beta} - h_{\alpha\alpha}h_{\beta\beta} + h_{\alpha\beta}^2 \notag \\
		&= \frac{1}{n-3}\left(\lambda_{\alpha} + \lambda_{\beta} - K_{\alpha 1} - K_{\beta 1} + H(h_{\alpha\alpha} + h_{\beta\beta}) - h_{\alpha\alpha}^2 - h_{\beta\beta}^2\right) - \frac{1}{2(n-3)} - h_{\alpha\alpha}h_{\beta\beta} \notag \\
		&\quad + W^{\Sigma}_{\alpha\beta} \notag \\
		&= \frac{1}{n-3}(\lambda_{\alpha} + \lambda_{\beta}) - \frac{1}{2(n-3)} - \frac{1}{n-3}(K_{\alpha 1} + K_{\beta 1}) \notag \\
		&\quad + \frac{1}{2(n-3)f}\left[\left(\frac{1}{2} - \lambda_{\alpha}\right) + \left(\frac{1}{2} - \lambda_{\beta}\right)\right] \notag \\
		&\quad - \frac{\left(\frac{1}{2} - \lambda_{\alpha}\right)^2 + \left(\frac{1}{2} - \lambda_{\beta}\right)^2}{(n-3)f} - \frac{\left(\frac{1}{2} - \lambda_{\alpha}\right)\left(\frac{1}{2} - \lambda_{\beta}\right)}{f} + W^{\Sigma}_{\alpha\beta}. \label{eq:K-alpha-beta}
	\end{aligned}
\end{equation*}
	By Lemma \ref{lem:ricci-gradient}, we have
	\begin{equation*}
	\begin{aligned}
		|\nabla \mathrm{Ric}|^2 &= \sum_{\substack{\alpha,\beta=2 \\ \alpha\neq\beta}}^n 2K_{\alpha\beta}\lambda_{\alpha}\lambda_{\beta} - \frac{n-2}{4} \notag \\
		&= \sum_{\substack{\alpha,\beta=2 \\ \alpha\neq\beta}}^n \left(\frac{2}{n-3}(\lambda_{\alpha} + \lambda_{\beta}) - \frac{1}{n-3}\right)\lambda_{\alpha}\lambda_{\beta} - \frac{n-2}{4} \notag \\
		&\quad - \sum_{\substack{\alpha,\beta=2 \\ \alpha\neq\beta}}^n \frac{2}{n-3}(K_{1\alpha} + K_{1\beta})\lambda_\alpha\lambda_\beta \notag \\
		&\quad + \sum_{\substack{\alpha,\beta=2 \\ \alpha\neq\beta}}^n \frac{1}{(n-3)f}\left[\left(\frac{1}{2} - \lambda_{\alpha}\right) + \left(\frac{1}{2} - \lambda_{\beta}\right)\right]\lambda_\alpha\lambda_\beta \notag \\
		&\quad - 2\sum_{\substack{\alpha,\beta=2 \\ \alpha\neq\beta}}^n \frac{\left(\frac{1}{2} - \lambda_{\alpha}\right)^2 + \left(\frac{1}{2} - \lambda_{\beta}\right)^2}{(n-3)f}\lambda_\alpha\lambda_\beta \notag \\
		&\quad - 2\sum_{\substack{\alpha,\beta=2 \\ \alpha\neq\beta}}^n \frac{\left(\frac{1}{2} - \lambda_{\alpha}\right)\left(\frac{1}{2} - \lambda_{\beta}\right)}{f}\lambda_\alpha\lambda_\beta + 2\sum_{\substack{\alpha,\beta=2 \\ \alpha\neq\beta}}^n W^{\Sigma}_{\alpha\beta}\lambda_\alpha\lambda_\beta. \label{eq:gradient-decomposition}
	\end{aligned}
\end{equation*}
	Next, we handle these terms one by one.
	
	\begin{claim}\label{claim:I}
		\begin{equation*}
			\begin{aligned}
			I &:= \sum_{\substack{\alpha,\beta=2 \\ \alpha\neq\beta}}^n \left(\frac{2}{n-3}(\lambda_{\alpha} + \lambda_{\beta}) - \frac{1}{n-3}\right)\lambda_{\alpha}\lambda_{\beta} - \frac{n-2}{4} \notag \\
			&\leq -\frac{3}{n-3}(\lambda_{1}+\lambda_{2}) + \frac{12}{n-3}(\lambda_{1}+\lambda_{2})^2 - \frac{12}{n-3}(\lambda_{1}+\lambda_{2})^3 + 2(n-2)(n-4)(\lambda_{1}+\lambda_{2})^{\frac{3}{2}}. 
		\end{aligned}
		\end{equation*}
	\end{claim}
	
	\begin{proof}[Proof of Claim \ref{claim:I}]
		We compute
	\begin{equation*}
	\begin{aligned}
			I &= \sum_{\substack{\alpha,\beta=2 \\ \alpha\neq\beta}}^n \left(\frac{2}{n-3}(\lambda_{\alpha} + \lambda_{\beta}) - \frac{1}{n-3}\right)\lambda_{\alpha}\lambda_{\beta} - \frac{n-2}{4} \\
			&= \frac{1}{n-3}\sum_{\substack{\alpha,\beta=2 \\ \alpha\neq\beta}}^n 4\lambda_\alpha^2 \lambda_\beta - \frac{1}{n-3}\sum_{\substack{\alpha,\beta=2 \\ \alpha\neq\beta}}^n \lambda_\alpha\lambda_\beta - \frac{n-2}{4} \\
			&= \frac{4}{n-3}\sum_{\alpha=2}^n \lambda_\alpha^2\left(\frac{n-2}{2} - \lambda_\alpha\right) - \frac{1}{n-3}\sum_{\alpha=2}^n\lambda_\alpha\left(\frac{n-2}{2} - \lambda_\alpha\right) - \frac{n-2}{4} \\
			&= \frac{4}{n-3}\left(\frac{n-2}{2} \cdot \frac{n-2}{4} - \sum_{\alpha=2}^n \lambda_\alpha^3\right) - \frac{1}{n-3}\left(\frac{n-2}{2} \cdot \frac{n-2}{2} - \frac{n-2}{4}\right) - \frac{n-2}{4} \\
			&= \frac{4}{n-3}\left(\frac{(n-2)^2}{8} - \sum_{\alpha=2}^n \lambda_\alpha^3\right) - \frac{1}{n-3} \cdot \frac{(n-2)(n-3)}{4} - \frac{n-2}{4} \\
			&= \frac{4}{n-3}\left(\frac{(n-2)^2}{8} - \sum_{\alpha=2}^n \lambda_\alpha^3\right) - \frac{n-2}{2},
		\end{aligned}
		\end{equation*}
		where we used $\sum_{\alpha=2}^n \lambda_\alpha = \frac{n-2}{2}$ and $\sum_{\alpha=2}^n \lambda_\alpha^2 = \frac{n-2}{4}$.
		
		Since for $\alpha = 3, \ldots, n$,
		\[
		\left(\lambda_\alpha - \frac{1}{2}\right)^3 = \lambda_\alpha^3 - \frac{3}{2}\lambda_\alpha^2 + \frac{3}{4}\lambda_\alpha - \frac{1}{8},
		\]
		we have
		\begin{align*}
			\sum_{\alpha=2}^n \lambda_\alpha^3 &= \lambda_2^3 + \sum_{\alpha=3}^n\left(\lambda_\alpha - \frac{1}{2}\right)^3 + \frac{3}{2}\sum_{\alpha=3}^n\lambda_\alpha^2 - \frac{3}{4}\sum_{\alpha=3}^n\lambda_\alpha + \frac{n-2}{8} \\
			&= \lambda_2^3 + \sum_{\alpha=3}^n\left(\lambda_\alpha - \frac{1}{2}\right)^3 + \frac{3}{2}\left(\frac{n-2}{4} - \lambda_2^2\right) - \frac{3}{4}\left(\frac{n-2}{2} - \lambda_2\right) + \frac{n-2}{8} \\
			&= \frac{n-2}{8} + \lambda_2^3 + \sum_{\alpha=3}^n\left(\lambda_\alpha - \frac{1}{2}\right)^3 - \frac{3}{2}\lambda_2^2 + \frac{3}{4}\lambda_2.
		\end{align*}
		Substituting this identity into $I$, we obtain
		\begin{align*}
			I &= \frac{4}{n-3}\left(\frac{(n-2)^2}{8} - \frac{n-2}{8} - \lambda_2^3 - \sum_{\alpha=3}^n\left(\lambda_\alpha - \frac{1}{2}\right)^3 + \frac{3}{2}\lambda_2^2 - \frac{3}{4}\lambda_2\right) - \frac{n-2}{2} \\
			&= -\frac{3}{n-3}\lambda_2 + \frac{6}{n-3}\lambda_2^2 - \frac{4}{n-3}\sum_{\alpha=3}^n\left(\lambda_\alpha - \frac{1}{2}\right)^3 - \frac{4}{n-3}\lambda_2^3 \\
			&\leq -\frac{3}{n-3}\lambda_2 + \frac{6}{n-3}\lambda_2^2 + 4\lambda_2^{\frac{3}{2}} - \frac{4}{n-3}\lambda_2^3,
		\end{align*}
		where in the last inequality we used $|\lambda_{\alpha} - \frac{1}{2}|^2 \leq \lambda_2(1-\lambda_2) \leq \lambda_2$ for $\alpha = 3, \ldots, n$.
	\end{proof}
	
	
	\begin{claim}\label{claim:II}
		\begin{equation*}\label{eq:claim-II}
			\begin{aligned}	
			II  :=& -\sum_{\substack{\alpha,\beta=2 \\ \alpha\neq\beta}}^n \frac{2}{n-3}(K_{1\alpha} + K_{1\beta})\lambda_\alpha\lambda_\beta\\
			\leq& 2\frac{|\nabla \mathrm{Ric}|^2}{f} + o(1)(\lambda_1+\lambda_2) + \frac{1}{2}K_{12}
				\end{aligned}
		\end{equation*}
		on $M \setminus D(a)$.
	\end{claim}
	
	\begin{proof}[Proof of Claim \ref{claim:II}]
		We rewrite this term as follows:
		\begin{align*}
			II &= -\sum_{\substack{\alpha,\beta=2 \\ \alpha\neq\beta}}^n (K_{1\alpha} + K_{1\beta})\lambda_\alpha\lambda_\beta \\
			&= -2\sum_{\substack{\alpha,\beta=2 \\ \alpha\neq\beta}}^n K_{1\alpha}\lambda_\alpha\lambda_\beta \\
			&= -2\sum_{\alpha=2}^n K_{1\alpha}\lambda_\alpha(R - \lambda_\alpha) \\
			&= -(n-2)\sum_{\alpha=2}^n K_{1\alpha}\lambda_\alpha + 2\sum_{\alpha=2}^n K_{1\alpha}\lambda_\alpha^2 \\
			&= -(n-2)\sum_{\alpha=2}^n K_{1\alpha}\lambda_\alpha + 2\left[\sum_{\alpha=2}^n K_{1\alpha}\left(\lambda_\alpha - \frac{1}{2}\right)^2 + \sum_{\alpha=2}^n K_{1\alpha}\lambda_\alpha - \frac{1}{4}\sum_{\alpha=2}^n K_{1\alpha}\right] \\
			&= (-n+4)\sum_{\alpha=2}^n K_{1\alpha}\lambda_\alpha + 2\sum_{\alpha=2}^n K_{1\alpha}\left(\lambda_\alpha - \frac{1}{2}\right)^2,
		\end{align*}
		where we used $\sum_{\alpha=2}^n K_{1\alpha} = R_{11} = 0$ in the last equality.
		
		By Lemma \ref{sectional-curvature} and \eqref{l345}, we have
		\begin{align*}
			-\sum_{\alpha=2}^n K_{1\alpha}\lambda_\alpha &= \frac{1}{f}\left[\left(\lambda_2 - \frac{1}{2}\right)^2\lambda_2 + \sum_{\alpha=3}^n\left(\lambda_\alpha - \frac{1}{2}\right)^2\lambda_\alpha\right] \\
			&\leq \frac{1}{f}\left[\left(\lambda_2 - \frac{1}{2}\right)^2\lambda_2 + \lambda_2(1-\lambda_2)\left(\frac{n-2}{2} - \lambda_2\right)\right] \\
			&\leq C\frac{\lambda_1+\lambda_2}{f}.
		\end{align*}
		It follows from \eqref{k1a} that
		\begin{align*}
			&2\sum_{\alpha=2}^n K_{1\alpha}\left(\lambda_\alpha - \frac{1}{2}\right)^2 \\
			&= 2K_{12}\left(\lambda_2 - \frac{1}{2}\right)^2 + 2\sum_{\alpha=3}^n K_{1\alpha}\left(\lambda_\alpha - \frac{1}{2}\right)^2 \\
			&\leq 2K_{12}\left(\lambda_2^2 - \lambda_2 + \frac{1}{4}\right) + \sum_{\alpha=3}^n K_{1\alpha}^2 + \sum_{\alpha=3}^n \left(\lambda_\alpha - \frac{1}{2}\right)^4 \\
			&\leq 2K_{12}(\lambda_2-1)\lambda_2 + \frac{1}{2}K_{12} + \sum_{\alpha=3}^n K_{1\alpha}^2 + \sum_{\alpha=3}^n \left(\lambda_\alpha - \frac{1}{2}\right)^4 \\
			&\leq K_{12}^2 + (\lambda_2-1)^2\lambda_{2}^2 + \frac{1}{2}K_{12} + \sum_{\alpha=3}^n K_{1\alpha}^2 + \left[\sum_{\alpha=3}^n\left(\lambda_\alpha - \frac{1}{2}\right)^2\right]^2 \\
			&= \sum_{\alpha=2}^n K_{1\alpha}^2 + \left[(\lambda_2-1)^2\lambda_{2}^2 + (1-\lambda_2)^2\lambda_{2}^2\right] + \frac{1}{2}K_{12}.
		\end{align*}
		Due to the boundedness of the Ricci curvature, we have
		\begin{equation}\label{eq:K1a-square}
		\begin{aligned}
			\sum_{\alpha=2}^n K_{1\alpha}^2 &= \frac{1}{f^2}\sum_{\alpha=2}^n\left[\nabla_{\nabla f}R_{\alpha\alpha} + \lambda_\alpha\left(\frac{1}{2} - \lambda_\alpha\right)\right]^2 \\
			&\leq \frac{2}{f^2}\left[\sum_{\alpha=2}^n\left(\nabla_{\nabla f}R_{\alpha\alpha}\right)^2 + \sum_{\alpha=2}^n\lambda_\alpha^2\left(\frac{1}{2} - \lambda_\alpha\right)^2\right] \\
			&\leq 2\frac{|\nabla \mathrm{Ric}|^2}{f} + \frac{2}{f^2}\sum_{\alpha=2}^n\lambda_\alpha^2\left(\frac{1}{2} - \lambda_\alpha\right)^2 \\
			&\leq 2\frac{|\nabla \mathrm{Ric}|^2}{f} + \frac{2}{f^2}\left[\lambda_2^2\left(\frac{1}{2} - \lambda_2\right)^2 + \sum_{\alpha=3}^n\lambda_\alpha^2\lambda_2(1-\lambda_2)\right] \\
			&\leq 2\frac{|\nabla \mathrm{Ric}|^2}{f} + \frac{2}{f^2}\lambda_2\left[\lambda_2\left(\frac{1}{2} - \lambda_2\right)^2 + \left(\frac{n-2}{4} - \lambda_2\right)(1-\lambda_2)\right] \\
			&\leq 2\frac{|\nabla \mathrm{Ric}|^2}{f} + \frac{C}{f^2}(\lambda_1+\lambda_2),
		\end{aligned}
	\end{equation}
Hence, combined with the fact $\lambda_1+\lambda_2 \to 0$ at infinity, we see
		on $M \setminus D(a)$ for some constant $C$ and sufficiently large $a$.
		\[
	2\sum_{\alpha=2}^n K_{1\alpha}\left(\lambda_\alpha - \frac{1}{2}\right)^2 \leq 2\frac{|\nabla \mathrm{Ric}|^2}{f} + o(1)(\lambda_1+\lambda_2) + \frac{1}{2}K_{12}.	
		\]
		In all, we obtain
		\[
		II \leq 2\frac{|\nabla \mathrm{Ric}|^2}{f} + o(1)(\lambda_1+\lambda_2) + \frac{1}{2}K_{12}.
		\]
	\end{proof}
	
	\begin{claim}\label{claim:III}
		\begin{equation*}\label{eq:claim-III}
			\begin{aligned}
				III &:= \sum_{\substack{\alpha,\beta=2 \\ \alpha\neq\beta}}^n \frac{1}{(n-3)f}\left[\left(\frac{1}{2} - \lambda_{\alpha}\right) + \left(\frac{1}{2} - \lambda_{\beta}\right)\right]\lambda_\alpha\lambda_\beta \\
				&\quad - 2\sum_{\substack{\alpha,\beta=2 \\ \alpha\neq\beta}}^n \frac{\left(\frac{1}{2} - \lambda_{\alpha}\right)^2 + \left(\frac{1}{2} - \lambda_{\beta}\right)^2}{(n-3)f}\lambda_\alpha\lambda_\beta \\
				&\quad - 2\sum_{\substack{\alpha,\beta=2 \\ \alpha\neq\beta}}^n \frac{\left(\frac{1}{2} - \lambda_{\alpha}\right)\left(\frac{1}{2} - \lambda_{\beta}\right)}{f}\lambda_\alpha\lambda_\beta \\
				&\leq \frac{C_2}{f}(\lambda_{1}+\lambda_{2}).
			\end{aligned}
		\end{equation*}
	\end{claim}
	
	\begin{proof}[Proof of Claim \ref{claim:III}]
		We rewrite the first two terms of $III$ as follows:
		\begin{align*}
			&\frac{2}{(n-3)f}\sum_{\substack{\alpha,\beta=2 \\ \alpha\neq\beta}}^n\left(\frac{1}{2} - \lambda_{\alpha}\right)\lambda_\alpha\lambda_\beta
			- \frac{4}{(n-3)f}\sum_{\substack{\alpha,\beta=2 \\ \alpha\neq\beta}}^n\left(\frac{1}{2} - \lambda_{\alpha}\right)^2\lambda_\alpha\lambda_\beta \\
			&= \frac{2}{(n-3)f}\sum_{\alpha=2}^n\left(\frac{1}{2} - \lambda_{\alpha}\right)\lambda_\alpha\left(\frac{n-2}{2} - \lambda_\alpha\right)
			- \frac{4}{(n-3)f}\sum_{\alpha=2}^n\left(\frac{1}{2} - \lambda_{\alpha}\right)^2\lambda_\alpha\left(\frac{n-2}{2} - \lambda_\alpha\right) \\
			&= \frac{2}{(n-3)f}\sum_{\alpha=2}^n\left(\frac{1}{2} - \lambda_{\alpha}\right)\lambda_\alpha\left(\frac{n-2}{2} - \lambda_\alpha\right)(-1 + 2\lambda_{\alpha}) \\
			&= \frac{4}{(n-3)f}\left[-\frac{n-3}{2}\sum_{\alpha=2}^n\lambda_\alpha\left(\frac{1}{2} - \lambda_{\alpha}\right)^2 + \frac{n-3}{4}\sum_{\alpha=2}^n\lambda_\alpha\left(\frac{1}{2} - \lambda_{\alpha}\right) + \sum_{\alpha=2}^n\lambda_\alpha^2\left(\frac{1}{2} - \lambda_{\alpha}\right)^2\right].
		\end{align*}
		Note that $\sum_{\alpha=2}^n\lambda_\alpha\left(\frac{1}{2} - \lambda_{\alpha}\right) = 0$ since
		\[
		\sum_{\alpha=2}^n\lambda_\alpha = R = 2|\mathrm{Ric}|^2 = 2\sum_{\alpha=2}^n\lambda_\alpha^2.
		\]
		It follows from \eqref{l345} and \eqref{l345'} that
		\begin{align*}
			\left|\sum_{\alpha=2}^n\lambda_\alpha\left(\frac{1}{2} - \lambda_{\alpha}\right)^2\right|
			&\leq \lambda_2\left(\frac{1}{2} - \lambda_2\right)^2 + \lambda_2(1-\lambda_2)\left(\frac{n-2}{2} - \lambda_2\right) \\
			&= \lambda_2\left[\left(\frac{1}{2} - \lambda_2\right)^2 + (1-\lambda_2)\left(\frac{n-2}{2} - \lambda_2\right)\right]
		\end{align*}
		and
		\begin{align*}
			\sum_{\alpha=2}^n\lambda_\alpha^2\left(\frac{1}{2} - \lambda_{\alpha}\right)^2
			&\leq \lambda_2^2\left(\frac{1}{2} - \lambda_2\right)^2 + \lambda_2(1-\lambda_2)\left(\frac{n-2}{4} - \lambda_2^2\right) \\
			&= \lambda_2\left[\lambda_2\left(\frac{1}{2} - \lambda_2\right)^2 + (1-\lambda_2)\left(\frac{n-2}{4} - \lambda_2^2\right)\right].
		\end{align*}
		Hence,
		\[
		\frac{2}{(n-3)f}\sum_{\substack{\alpha,\beta=2 \\ \alpha\neq\beta}}^n\left(\frac{1}{2} - \lambda_{\alpha}\right)\lambda_\alpha\lambda_\beta
		- \frac{4}{(n-3)f}\sum_{\substack{\alpha,\beta=2 \\ \alpha\neq\beta}}^n\left(\frac{1}{2} - \lambda_{\alpha}\right)^2\lambda_\alpha\lambda_\beta
		\leq \frac{C}{f}(\lambda_{1}+\lambda_{2})
		\]
		for some constant $C$, since the Ricci curvature is nonnegative and bounded.
		
		We consider the third term of $III$ as follows:
		\begin{align*}
			&-\sum_{\substack{\alpha,\beta=2 \\ \alpha\neq\beta}}^n\left(\frac{1}{2} - \lambda_{\alpha}\right)\left(\frac{1}{2} - \lambda_{\beta}\right)\lambda_\alpha\lambda_\beta \\
			&= -\frac{1}{2}\sum_{\substack{\alpha,\beta=2 \\ \alpha\neq\beta}}^n\left(\frac{1}{2} - \lambda_{\alpha}\right)\lambda_\alpha\lambda_\beta
			+ \sum_{\substack{\alpha,\beta=2 \\ \alpha\neq\beta}}^n\left(\frac{1}{2} - \lambda_{\alpha}\right)\lambda_\alpha\lambda_\beta^2 \\
			&= -\frac{1}{2}\sum_{\alpha=2}^n\left(\frac{1}{2} - \lambda_{\alpha}\right)\lambda_\alpha\left(\frac{n-2}{2} - \lambda_\alpha\right)
			+ \sum_{\alpha=2}^n\left(\frac{1}{2} - \lambda_{\alpha}\right)\lambda_\alpha\left(\frac{n-2}{4} - \lambda_\alpha^2\right) \\
			&= -\frac{n-2}{4}\sum_{\alpha=2}^n\left(\frac{1}{2} - \lambda_{\alpha}\right)\lambda_\alpha
			+ \frac{1}{2}\sum_{\alpha=2}^n\left(\frac{1}{2} - \lambda_{\alpha}\right)\lambda_\alpha^2
			+ \frac{n-2}{4}\sum_{\alpha=2}^n\left(\frac{1}{2} - \lambda_{\alpha}\right)\lambda_\alpha
			- \sum_{\alpha=2}^n\left(\frac{1}{2} - \lambda_{\alpha}\right)\lambda_\alpha^3 \\
			&= \sum_{\alpha=2}^n\lambda_\alpha^2\left(\frac{1}{2} - \lambda_{\alpha}\right)^2 \\
			&\leq \lambda_2\left[\lambda_2\left(\frac{1}{2} - \lambda_2\right)^2 + (1-\lambda_2)\left(\frac{n-2}{4} - \lambda_2^2\right)\right],
		\end{align*}
		which implies that
		\[
		-\frac{2}{f}\sum_{\substack{\alpha,\beta=2 \\ \alpha\neq\beta}}^n\left(\frac{1}{2} - \lambda_{\alpha}\right)\left(\frac{1}{2} - \lambda_{\beta}\right)\lambda_\alpha\lambda_\beta
		\leq \frac{C}{f}(\lambda_{1}+\lambda_{2})
		\]
		for some constant $C$, since the Ricci curvature is nonnegative and bounded. Therefore,
		\[
		III \leq \frac{C}{f}(\lambda_{1}+\lambda_{2}).
		\]
	\end{proof}
	
	
	\begin{claim}\label{claim:IV}
		\begin{equation*}\label{eq:claim-IV}
		\begin{aligned}
			IV := &2\sum_{\substack{\alpha,\beta=2 \\ \alpha\neq\beta}}^n W^{\Sigma}_{\alpha\beta}\lambda_\alpha\lambda_\beta\\
			 \leq& \frac{5n-8}{4}\left[\epsilon(\lambda_1+\lambda_2) + \frac{1}{\epsilon}|W^{\Sigma}|^2\right]
		\end{aligned}
		\end{equation*}
		for any constant $\epsilon > 0$ satisfying $\frac{5n-8}{4}\epsilon \leq \frac{1}{4(n-3)}$.
	\end{claim}
	
	\begin{proof}[Proof of Claim \ref{claim:IV}]
		We rewrite the term $IV$ as follows:
		\begin{align*}
			IV &= 4W^{\Sigma}_{23}\lambda_{2}\lambda_{3} + 4W^{\Sigma}_{24}\lambda_{2}\lambda_{4} + \cdots + 4W^{\Sigma}_{2n}\lambda_{2}\lambda_{n} \\
			&\quad + 4W^{\Sigma}_{34}\lambda_{3}\lambda_{4} + 4W^{\Sigma}_{35}\lambda_{3}\lambda_{5} + \cdots + 4W^{\Sigma}_{(n-1)n}\lambda_{n-1}\lambda_{n} \\
			&= 4\lambda_{2}(W^{\Sigma}_{23}\lambda_{3} + W^{\Sigma}_{24}\lambda_{4} + \cdots + W^{\Sigma}_{2n}\lambda_{n}) \\
			&\quad + 2\lambda_3(W^{\Sigma}_{34}\lambda_{4} + W^{\Sigma}_{35}\lambda_{5} + \cdots + W^{\Sigma}_{3n}\lambda_{n}) \\
			&\quad + 2\lambda_4(W^{\Sigma}_{43}\lambda_{3} + W^{\Sigma}_{45}\lambda_{5} + \cdots + W^{\Sigma}_{4n}\lambda_{n}) \\
			&\quad + \cdots + 2\lambda_n(W^{\Sigma}_{n3}\lambda_{3} + W^{\Sigma}_{n4}\lambda_{4} + \cdots + W^{\Sigma}_{n(n-1)}\lambda_{n-1}).
		\end{align*}
		Using the fact $\sum_{\alpha=2}^n W^{\Sigma}_{3\alpha} = 0$, we have
		\begin{align*}
			&2\lambda_3(W^{\Sigma}_{34}\lambda_{4} + W^{\Sigma}_{35}\lambda_{5} + \cdots + W^{\Sigma}_{3n}\lambda_{n}) \\
			&= 2\lambda_3\left[W^{\Sigma}_{34}\left(\lambda_{4} - \frac{1}{2}\right) + W^{\Sigma}_{35}\left(\lambda_{5} - \frac{1}{2}\right) + \cdots + W^{\Sigma}_{3n}\left(\lambda_{n} - \frac{1}{2}\right) - \frac{1}{2}W^{\Sigma}_{32}\right] \\
			&\leq 2\left(|W^{\Sigma}_{34}|^2 + |W^{\Sigma}_{35}|^2 + \cdots + |W^{\Sigma}_{3n}|^2\right)^{\frac{1}{2}}\left[\left(\lambda_{4} - \frac{1}{2}\right)^2 + \left(\lambda_{5} - \frac{1}{2}\right)^2 + \cdots + \left(\lambda_{n} - \frac{1}{2}\right)^2\right]^{\frac{1}{2}} \\
			&\quad - W^{\Sigma}_{23}\lambda_{3} \\
			&\leq 2|W^{\Sigma}|\lambda_2^{\frac{1}{2}} - W^{\Sigma}_{23}\lambda_{3} \\
			&\leq \epsilon\lambda_2 + \frac{1}{\epsilon}|W^{\Sigma}|^2 - W^{\Sigma}_{23}\lambda_{3}
		\end{align*}
		for any $\epsilon > 0$, where we used the fact
		\[
		\left(\lambda_{4} - \frac{1}{2}\right)^2 + \left(\lambda_{5} - \frac{1}{2}\right)^2 + \cdots + \left(\lambda_{n} - \frac{1}{2}\right)^2 \leq \lambda_2(1-\lambda_2) \leq \lambda_2
		\]
		in the second inequality. Similarly,
		\[
		2\lambda_4(W^{\Sigma}_{34}\lambda_{3} + W^{\Sigma}_{45}\lambda_{5} + \cdots + W^{\Sigma}_{3n}\lambda_{n}) \leq \epsilon\lambda_2 + \frac{1}{\epsilon}|W^{\Sigma}|^2 - W^{\Sigma}_{24}\lambda_{4},
		\]
		\[
		2\lambda_n(W^{\Sigma}_{3n}\lambda_{3} + W^{\Sigma}_{4n}\lambda_{4} + \cdots + W^{\Sigma}_{(n-1)n}\lambda_{n-1}) \leq \epsilon\lambda_2 + \frac{1}{\epsilon}|W^{\Sigma}|^2 - W^{\Sigma}_{2n}\lambda_{n}.
		\]
		Therefore,
		\[
		IV \leq 4\lambda_{2}\sum_{\alpha=3}^n W^{\Sigma}_{2\alpha}\lambda_{\alpha} + (n-2)\left(\epsilon\lambda_2 + \frac{1}{\epsilon}|W^{\Sigma}|^2\right) - \sum_{\alpha=3}^n W^{\Sigma}_{2\alpha}\lambda_{\alpha}.
		\]
		
		Next, we deal with the term $\sum_{\alpha=3}^n W^{\Sigma}_{2\alpha}\lambda_{\alpha}$. Since the Weyl curvature is trace-free, we have
		\[
		\sum_{\alpha=3}^n W^{\Sigma}_{2\alpha} = \sum_{\alpha=3}^n W^{\Sigma}_{2\alpha 2\alpha} = 0,
		\]
		and then
		\begin{equation}\label{eq:w2a}
			\begin{aligned}
				&|W^{\Sigma}_{23}\lambda_{3} + W^{\Sigma}_{24}\lambda_{4} + \cdots + W^{\Sigma}_{2n}\lambda_{n}| \\
				&= \left|W^{\Sigma}_{23}\left(\lambda_{3} - \frac{1}{2}\right) + W^{\Sigma}_{24}\left(\lambda_{4} - \frac{1}{2}\right) + \cdots + W^{\Sigma}_{2n}\left(\lambda_{n} - \frac{1}{2}\right)\right| \\
				&\leq |W^{\Sigma}|\left[\left(\lambda_{3} - \frac{1}{2}\right)^2 + \left(\lambda_{4} - \frac{1}{2}\right)^2 + \cdots + \left(\lambda_{n} - \frac{1}{2}\right)^2\right]^{\frac{1}{2}} \\
				&= |W^{\Sigma}|[\lambda_{2}(1-\lambda_2)]^{\frac{1}{2}} \\
				&\leq |W^{\Sigma}|\lambda_{2}^{\frac{1}{2}} \\
				&\leq \frac{1}{2}\left(\epsilon\lambda_2 + \frac{1}{\epsilon}|W^{\Sigma}|^2\right).
			\end{aligned}
		\end{equation}
		Therefore,
		\begin{align*}
			IV &\leq 2\lambda_{2}\left(\epsilon\lambda_2 + \frac{1}{\epsilon}|W^{\Sigma}|^2\right) + (n-2)\left(\epsilon\lambda_2 + \frac{1}{\epsilon}|W^{\Sigma}|^2\right) + \frac{1}{2}\left(\epsilon\lambda_2 + \frac{1}{\epsilon}|W^{\Sigma}|^2\right) \\
			&\leq \frac{2n+1}{2}\left(\epsilon\lambda_2 + \frac{1}{\epsilon}|W^{\Sigma}|^2\right),
		\end{align*}
		due to the fact $(n-1)\lambda_{2} \leq R = \frac{n-2}{2}$.
	\end{proof}
	
	Consequently, from Claims \ref{claim:I}--\ref{claim:IV} and the fact that $\lambda_1+\lambda_2 \to 0$ at infinity, we obtain
	\begin{align*}
		|\nabla \mathrm{Ric}|^2 &= I + II + III + IV \\
		&\leq -\frac{3}{n-3}(\lambda_{1}+\lambda_{2}) + \frac{6}{n-3}(\lambda_{1}+\lambda_{2})^2 - \frac{4}{n-3}(\lambda_{1}+\lambda_{2})^3 + 4(\lambda_{1}+\lambda_{2})^{\frac{3}{2}} \\
		&\quad + 2\frac{|\nabla \mathrm{Ric}|^2}{f} + o(1)(\lambda_1+\lambda_2) + \frac{1}{2}K_{12} + \frac{C}{f}(\lambda_{1}+\lambda_{2}) \\
		&\quad + \frac{2n+1}{2}\left(\epsilon\lambda_2 + \frac{1}{\epsilon}|W^{\Sigma}|^2\right) \\
		&\leq 2\frac{|\nabla \mathrm{Ric}|^2}{f} - \frac{2}{n-3}(\lambda_1+\lambda_2) + \frac{1}{2}K_{12} + C|W^{\Sigma}|^2
	\end{align*}
	on $M \setminus D(a)$ for some constant $C$, sufficiently large $a$, and small $\epsilon$ satisfying $\frac{2n+1}{2}\epsilon \leq \frac{1}{4(n-3)}$.
	
	Hence,
	\[
	|\nabla \mathrm{Ric}|^2 \leq -\frac{1}{n-3}(\lambda_1+\lambda_2) + K_{12} + C|W^{\Sigma}|^2
	\]
	on $M \setminus D(a)$, where $a > 0$ is sufficiently large. This completes the proof of Proposition \ref{prop:key-estimate}.
\end{proof}

\section{Proof of Theorem \ref{thm:main}}\label{sec:proof-main}

By the standard decomposition of the Riemann curvature operator, for any $(M^n, g)$ with $n \geq 4$,
\begin{equation*}\label{eq:curvature-decomposition}
	|Rm|^2 = \frac{4}{n-2}|Ric|^2 - \frac{2}{(n-1)(n-2)}R^2 + |W|^2.
\end{equation*}

Since $Ric^\Sigma = Ric - \frac{\nabla_{\nabla f}Ric}{f} \geq 0$ by \eqref{ricci-curvature-level-set} and assumption (1) in Theorem \ref{thm:main}, the Cheeger--Gromoll splitting theorem \cite{Cheeger-Gromoll} implies that $\Sigma(s)$ isometrically splits as $\mathbb{S}^{n-2}(\Sigma(s)) \times \mathbb{S}^1(\Sigma(s))$. Consequently,
\[
|Rm^{\Sigma(s)}| = |Rm^{\mathbb{S}^{n-2}(\Sigma(s))}|, \quad |Ric^{\Sigma(s)}| = |Ric^{\mathbb{S}^{n-2}(\Sigma(s))}|, \quad R^{\Sigma(s)} = R^{\mathbb{S}^{n-2}(\Sigma(s))}.
\]

Since
\begin{equation*}\label{eq:curvature-level-set}
	|Rm^{\Sigma(s)}|^2 = \frac{4}{n-3}|Ric^{\Sigma(s)}|^2 - \frac{2}{(n-2)(n-3)}(R^{\Sigma(s)})^2 + |W^{\Sigma(s)}|^2
\end{equation*}
and
\begin{equation*}\label{eq:curvature-sphere}
	|Rm^{\mathbb{S}^{n-2}(\Sigma(s))}|^2 = \frac{4}{n-4}|Ric^{\mathbb{S}^{n-2}(\Sigma(s))}|^2 - \frac{2}{(n-3)(n-4)}(R^{\mathbb{S}^{n-2}(\Sigma(s))})^2 + |W^{\mathbb{S}^{n-2}(\Sigma(s))}|^2,
\end{equation*}
we arrive at the following identity:
\begin{align*}
	&\frac{4}{n-3}\left(\frac{n-2}{4} + \left|\frac{\nabla_{\nabla f}Ric}{f}\right|^2\right) - \frac{2}{(n-2)(n-3)} \cdot \left(\frac{n-2}{2}\right)^2 + |W^{\Sigma(s)}|^2 \\
	&= \frac{4}{n-4}\left(\frac{n-2}{4} + \left|\frac{\nabla_{\nabla f}Ric}{f}\right|^2\right) - \frac{2}{(n-3)(n-4)} \cdot \left(\frac{n-2}{2}\right)^2 + |W^{\mathbb{S}^{n-2}(\Sigma(s))}|^2.
\end{align*}
Through careful calculation, it is easy to derive that
\begin{equation*}\label{eq:weyl-comparison}
	\begin{aligned}
		|W^{\Sigma(s)}|^2 &= |W^{\mathbb{S}^{n-2}(\Sigma(s))}|^2 + \left(\frac{4}{n-4} - \frac{4}{n-3}\right)\left|\frac{\nabla_{\nabla f}Ric}{f}\right|^2 \\
		&\leq |W^{\mathbb{S}^{n-2}(\Sigma(s))}|^2 + \frac{4}{(n-3)(n-4)}\frac{|\nabla Ric|^2}{f}.
	\end{aligned}
\end{equation*}

Since $(M^n, g, f)$ smoothly converges to $\mathbb{R}^2 \times \mathbb{S}^{n-2}$, the level sets of $f$ converge to $\mathbb{R} \times \mathbb{S}^{n-2}$, i.e., $\mathbb{S}^{n-2}(\Sigma(s))$ converges to the standard $\mathbb{S}^{n-2}$ with scalar curvature $\frac{n-2}{2}$. Then we can apply Theorem \ref{thm:weyl-curvature-estimate} to obtain
\begin{equation*}\label{eq:weyl-sphere-estimate}
	\int_{\mathbb{S}^{n-2}(\Sigma(s))}|W^{\mathbb{S}^{n-2}(\Sigma(s))}|^2 \leq C\int_{\mathbb{S}^{n-2}(\Sigma(s))}|(\mathring{Ric})^{\mathbb{S}^{n-2}(\Sigma(s))}|^2.
\end{equation*}
Now we can estimate the integral of the $L^2$ norm of the Weyl curvature as follows:
\begin{equation}\label{eq:weyl-integral-estimate}
	\begin{aligned}
	\int_{\Sigma(s)}|W^{\Sigma(s)}|^2 &= \int_{\mathbb{S}^{n-2}(\Sigma(s)) \times \mathbb{S}^1(\Sigma(s))}|W^{\mathbb{S}^{n-2}(\Sigma(s)) \times \mathbb{S}^1(\Sigma(s))}|^2 \notag \\
	&\leq \int_{\mathbb{S}^1(\Sigma(s))}\int_{\mathbb{S}^{n-2}(\Sigma(s))}\left(|W^{\mathbb{S}^{n-2}(\Sigma(s))}|^2 + \frac{4}{(n-3)(n-4)}\frac{|\nabla Ric|^2}{f}\right) \notag \\
	&\leq C\int_{\mathbb{S}^1(\Sigma(s))}\int_{\mathbb{S}^{n-2}(\Sigma(s))}\left(|(\mathring{Ric})^{\mathbb{S}^{n-2}(\Sigma(s))}|^2 + \frac{|\nabla Ric|^2}{f}\right) \notag \\
	&= C\int_{\mathbb{S}^1(\Sigma(s))}\int_{\mathbb{S}^{n-2}(\Sigma(s))} \left(|Ric^{\mathbb{S}^{n-2}(\Sigma(s))}|^2 - \frac{1}{n-2}(R^{\mathbb{S}^{n-2}(\Sigma(s))})^2 + \frac{|\nabla Ric|^2}{f}\right) \notag \\
	&= C \int_{\Sigma(s)}\left(|Ric^{\Sigma(s)}|^2 - \frac{1}{n-2}(R^{\Sigma(s)})^2 + \frac{|\nabla Ric|^2}{f}\right) \notag \\
	&= C\int_{\Sigma(s)} \left(\frac{n-2}{4} + \left|\frac{\nabla_{\nabla f}Ric}{f}\right|^2 - \frac{1}{n-2} \cdot \frac{(n-2)^2}{4} + \frac{|\nabla Ric|^2}{f}\right) \notag \\
	&\leq C\int_{\Sigma(s)} \frac{|\nabla Ric|^2}{f}.
\end{aligned}
\end{equation}
\begin{prop}\label{prop:ricci-gradient-estimate}
	Let $(M^n, g, f)$ be an $n$-dimensional shrinking gradient Ricci soliton with $R = \frac{n-2}{2}$. Suppose the assumptions in Theorem \ref{thm:main} hold. Then
	\begin{equation*}\label{eq:ricci-gradient-integral}
		\int_{\Sigma(s)}|\nabla Ric|^2 \leq -\frac{1.1}{n-3}\int_{\Sigma(s)}(\lambda_1+\lambda_2) + \frac{1.1}{s}\int_{\Sigma(s)}\langle\nabla (\lambda_1+\lambda_2), \nabla f \rangle
	\end{equation*}
	for sufficiently large $s$ almost everywhere.
\end{prop}

\begin{remark}\label{rem:lipshitz}
	The function $\lambda_1+\lambda_2$ is only Lipschitz continuous and differentiable almost everywhere. Here $\langle\nabla (\lambda_1+\lambda_2), \nabla f \rangle$ can be understood as follows: At any $p \in M$, choose $\{e_1, e_2\}$ to be the eigenvectors corresponding to the eigenvalues $\{\lambda_1, \lambda_2\}$, where $\lambda_1 \leq \lambda_2$ are the smallest two eigenvalues of the Ricci curvature. Take parallel transport along all geodesics starting from $p$; then we obtain two smooth vector fields $\{e_1, e_2\}$ in a neighborhood of $p$ such that $e_1(p) = e_1$, $e_2(p) = e_2$. It is easy to check that
	\[
	(\nabla_{\nabla f}Ric)(e_1, e_1) + (\nabla_{\nabla f}Ric)(e_2, e_2) = \nabla f \cdot (Ric(e_1, e_1) + Ric(e_2, e_2)) = \langle\nabla f, \lambda_1+\lambda_2\rangle
	\]
	if $\lambda_1+\lambda_2$ is differentiable at $p$.
\end{remark}

\begin{proof}[Proof of Proposition \ref{prop:ricci-gradient-estimate}]
	Recall from Proposition \ref{prop:key-estimate} that
	\[
	|\nabla Ric|^2 \leq -\frac{1}{n-3}(\lambda_1+\lambda_2) + K_{12} + C|W^{\Sigma(s)}|^2.
	\]
	Integrating over $\Sigma(s)$ and using \eqref{eq:weyl-integral-estimate} and Lemma \ref{sectional-curvature}, we obtain
	\begin{align*}
		\int_{\Sigma(s)}|\nabla Ric|^2 &\leq -\frac{1}{n-3}\int_{\Sigma(s)}(\lambda_1+\lambda_2) + C\int_{\Sigma(s)} \frac{|\nabla Ric|^2}{f} \\
		&\quad + \int_{\Sigma(s)} \frac{\nabla f \cdot \nabla(\lambda_1+\lambda_2) + \frac{1}{2}(\lambda_1+\lambda_2) - (\lambda_1+\lambda_2)^2}{f}.
	\end{align*}
	Since $\lambda_1+\lambda_2$ is sufficiently small and $f$ is sufficiently large outside a compact set, by absorption we derive that
	\[
	\int_{\Sigma(s)}|\nabla Ric|^2 \leq -\frac{1.1}{n-3}\int_{\Sigma(s)}(\lambda_1+\lambda_2) + \frac{1.1}{s}\int_{\Sigma(s)}\langle\nabla (\lambda_1+\lambda_2), \nabla f \rangle
	\]
	for sufficiently large $s$.
\end{proof}

\begin{prop}\label{prop:monotonicity}
	Let $(M^n, g, f)$ be an $n$-dimensional shrinking gradient Ricci soliton with $R = \frac{n-2}{2}$. Suppose the assumptions in Theorem \ref{thm:main} hold. Then
	\begin{equation*}\label{eq:monotonicity}
		\int_{\Sigma(s)} \langle\nabla (\lambda_1+\lambda_2), \nabla f \rangle \, d\sigma_{\Sigma(s)} \leq 0
	\end{equation*}
	for sufficiently large $s$ almost everywhere.
\end{prop}

\begin{proof}
	For this purpose, we consider the following one-parameter family of diffeomorphisms:
	\begin{equation*}\label{eq:diffeomorphism}
		\begin{cases}
			\dfrac{\partial F}{\partial s} = \dfrac{\nabla f}{|\nabla f|^2}, \\[6pt]
			F(x, a) = x \in \Sigma(a).
		\end{cases}
	\end{equation*}
	Then $\frac{\partial f}{\partial s} = \langle\nabla f, \frac{\nabla f}{|\nabla f|^2}\rangle = 1$, and the advantage of $F$ is that it maps level sets of $f$ to other level sets; in particular, $f(F(x, s)) = s$.
	
	Suppose $\{x_1, x_2, \ldots, x_{n-1}\}$ is a local coordinate chart of $\Sigma(a)$. On $\Sigma(s)$, let $g(s)(\frac{\partial}{\partial x_i}, \frac{\partial}{\partial x_j}) := g(\frac{\partial F}{\partial x_i}, \frac{\partial F}{\partial x_j})$ and $d\sigma_{\Sigma(s)} = \sqrt{\det(g_{ij})}\,dx$, where $dx = dx_1 \wedge dx_2 \wedge \cdots \wedge dx_{n-1}$.
	
	Next we compute the derivative of $d\sigma_{\Sigma(s)}$:
	\begin{align*}
		\frac{\partial}{\partial s}d\sigma_{\Sigma(s)} &= \frac{\partial}{\partial s}\sqrt{\det(g_{ij})}\,dx \\
		&= \frac{1}{2} \cdot 2g^{ij}\left\langle \nabla_{\frac{\partial F}{\partial x_i}}\frac{\partial F}{\partial s}, \frac{\partial F}{\partial x_j}\right\rangle d\sigma_{\Sigma(s)} \\
		&= g^{ij}\left\langle \nabla_{\frac{\partial F}{\partial x_i}}\frac{\nabla f}{|\nabla f|^2}, \frac{\partial F}{\partial x_j}\right\rangle d\sigma_{\Sigma(s)} \\
		&= \frac{1}{|\nabla f|^2}g^{ij}\nabla^2 f\left(\frac{\partial F}{\partial x_i}, \frac{\partial F}{\partial x_j}\right)d\sigma_{\Sigma(s)} \\
		&= \frac{1}{|\nabla f|^2} g^{ij}\left(\frac{1}{2}g\left(\frac{\partial F}{\partial x_i}, \frac{\partial F}{\partial x_j}\right) - Ric\left(\frac{\partial F}{\partial x_i}, \frac{\partial F}{\partial x_j}\right)\right)d\sigma_{\Sigma(s)} \\
		&= \frac{1}{|\nabla f|^2}\left(\frac{n-1}{2} - \frac{n-2}{2}\right)d\sigma_{\Sigma(s)} = \frac{1}{2s}d\sigma_{\Sigma(s)}.
	\end{align*}
	Hence, it is easy to check that
	\[
	\frac{\partial}{\partial s}\left(\frac{1}{\sqrt{s}}d\sigma_{\Sigma(s)}\right) = \left(-\frac{1}{2}s^{-\frac{3}{2}} + \frac{1}{\sqrt{s}} \cdot \frac{1}{2s}\right) d\sigma_{\Sigma(s)} = 0.
	\]
	Define
	\[
	I(s) = \int_{\Sigma(s)} (\lambda_1+\lambda_2) \cdot \frac{1}{\sqrt{s}} \, d\sigma_{\Sigma(s)}.
	\]
	It is Lipschitz continuous, hence differentiable almost everywhere, and we can compute its derivative as follows:
	\begin{align*}
		I'(s) &= \frac{d}{ds}\int_{\Sigma(s)} (\lambda_1+\lambda_2) \cdot \frac{1}{\sqrt{s}} \, d\sigma_{\Sigma(s)} \\
		&= \int_{\Sigma(s)} \left\langle \nabla (\lambda_1+\lambda_2), \frac{\nabla f}{|\nabla f|^2}\right\rangle \frac{1}{\sqrt{s}}\,d\sigma_{\Sigma(s)} + \int_{\Sigma(s)} (\lambda_1+\lambda_2)\frac{\partial}{\partial s}\left(\frac{1}{\sqrt{s}}d\sigma_{\Sigma(s)}\right) \\
		&= \int_{\Sigma(s)} \left\langle \nabla (\lambda_1+\lambda_2), \frac{\nabla f}{|\nabla f|^2}\right\rangle \frac{1}{\sqrt{s}}\,d\sigma_{\Sigma(s)} \\
		&= \frac{1}{s^{\frac{3}{2}}}\int_{\Sigma(s)} \langle \nabla (\lambda_1+\lambda_2), \nabla f\rangle \, d\sigma_{\Sigma(s)},
	\end{align*}
	where we used $|\nabla f|^2 = s$ in the last equality.
	
	Moreover, since $I(s)$ tends to zero as $s \to \infty$ by assumption, there exists sufficiently large $b > a$ such that $I'(b) \leq 0$, i.e.,
	\[
	\int_{\Sigma(b)} \langle \nabla(\lambda_1+\lambda_2), \nabla f \rangle \, d\sigma_{\Sigma(b)} \leq 0.
	\]
	Finally, we claim that
	\[
	\int_{\Sigma(s)} \langle\nabla (\lambda_1+\lambda_2), \nabla f \rangle \, d\sigma_{\Sigma(s)} \leq 0
	\]
	for almost every $s \geq b$.
	
	Indeed, if not, assume there exists some $c > b$ such that
	\[
	\int_{\Sigma(c)} \langle\nabla (\lambda_1+\lambda_2), \nabla f \rangle \, d\sigma_{\Sigma(c)} > 0.
	\]
	Similarly, because $I(s)$ tends to zero as $s \to \infty$ by assumption, there exists sufficiently large $d > c$ such that $I'(d) < 0$, i.e.,
	\[
	\int_{\Sigma(d)} \langle\nabla (\lambda_1+\lambda_2), \nabla f \rangle \, d\sigma_{\Sigma(d)} < 0.
	\]
	Then it follows from Proposition \ref{prop:ricci-gradient-estimate} that
	\begin{align*}
		\int_{\Sigma(d)}|\nabla Ric|^2 \, d\sigma_{\Sigma(d)} &\leq -\frac{1.1}{n-3}\int_{\Sigma(d)} (\lambda_1+\lambda_2)\, d\sigma_{\Sigma(d)} + \frac{1.1}{d}\int_{\Sigma(d)} \langle\nabla (\lambda_1+\lambda_2), \nabla f \rangle \, d\sigma_{\Sigma(d)} \\
		&<< 0,
	\end{align*}
	which is a contradiction. This completes the proof of Proposition \ref{prop:monotonicity}.
\end{proof}

\begin{proof}[Proof of Theorem \ref{thm:main}]
	Applying Propositions \ref{prop:ricci-gradient-estimate} and \ref{prop:monotonicity}, we infer
	\[
	\int_{\Sigma(s)}|\nabla Ric|^2 \, d\sigma_{\Sigma(s)} = 0 \quad \text{and} \quad \int_{\Sigma(s)} (\lambda_1+\lambda_2)\, d\sigma_{\Sigma(s)} = 0
	\]
	for sufficiently large $s$ almost everywhere, due to the nonnegativity of $\lambda_1+\lambda_2$. Thus $\nabla Ric = 0$ and $\lambda_1+\lambda_2 = 0$ on $M \setminus D(s)$ by the continuity of $\nabla Ric$ and $\lambda_1+\lambda_2$.
	Since the Ricci curvature is nonnegative, we obtain
	\[
	\lambda_1 = \lambda_2 \equiv 0 \quad \text{and} \quad \lambda_3 = \cdots = \lambda_n \equiv \frac{1}{2}.
	\]
	Due to the analyticity of the gradient Ricci soliton, $\nabla Ric = 0$ on $M$.
	
	Finally, the de Rham splitting theorem implies that $(M^n, g, f)$ is isometric to a finite quotient of $\mathbb{R}^2 \times {N}^{n-2}$, where ${N}^{n-2}$ is an $(n-2)$-dimensional Einstein manifold with Einstein constant $\frac{1}{2}$. Since $(M^n, g, f)$ converges smoothly to $\mathbb{R}^2 \times \mathbb{S}^{n-2}$, the manifold ${N}^{n-2}$ must be isometric to $\mathbb{S}^{n-2}$. This completes the proof of Theorem \ref{thm:main}.
\end{proof}

{\bf Acknowledgments}.
The third author would like to thank Professor Xi-nan Ma for his constant encouragement.

\end{document}